\documentclass[12pt]{article}

\usepackage{amssymb}
\usepackage{amsmath,amsthm}
\usepackage[utf8]{inputenc}
\usepackage[dvips]{graphicx}
\usepackage{hyperref}
\usepackage{enumerate}
\usepackage{color}
\usepackage{tikz,ifthen}
\usepackage[T1]{fontenc}
\usepackage{float}

\usepackage{lineno}

\newcommand{\cp}{\textrm{cp}}

\newtheorem{remark}{Remark}

\newtheorem{theorem}[remark]{Theorem}
\newtheorem{observation}[remark]{Observation}
\newtheorem{proposition}[remark]{Proposition}
\newtheorem{corollary}[remark]{Corollary}

\newtheorem{problem}{Problem}

\hypersetup{colorlinks=true}
\hypersetup{colorlinks=true, linkcolor=blue, citecolor=blue,urlcolor=blue}
\newcommand{\rmv}[1]{}

\begin{document}

\title{On the Independence Number of the Modular Product}
\author{Tanja Dravec$^{(1,3)}$,  Iztok Peterin$^{(2,3)}$\\
%EndAName
\\
$^{(1)}${\small Faculty of Natural Sciences and Mathematics}\\
{\small University of Maribor,} {\small Koro\v{s}ka cesta 160, 2000 Maribor,
Slovenia.}\\
$^{(2)}$ {\small Faculty of Electrical Engineering and Computer Science}\\
{\small University of Maribor,} {\small Koro\v{s}ka cesta 46, 2000 Maribor,
Slovenia.} \\
$^{(3)}$ {\small Institute of Mathematics, Physics and Mechanics}\\
{\small Jadranska ulica 19, 1000 Ljubljana, Slovenia.} \\
{\small \texttt{e-mail:} \textit{tanja.dravec\@@um.si}; \textit{iztok.peterin\@@um.si}} \\
}

\maketitle

\date{}

\begin{abstract}
The \emph{modular product} $G\diamond H$ of graphs $G$ and $H$ is a graph on vertex set $V(G)\times V(H)$. Two vertices $(g,h)$ and $(g',h')$ of $G\diamond H$ are adjacent if $g=g'$ and $hh'\in E(H)$, or $gg'\in E(G)$ and $h=h'$, or $gg'\in E(G)$ and $hh'\in E(H)$, or (for $g\neq g'$ and $h\neq h'$) $gg'\notin E(G)$ and $hh'\notin E(H)$. The independence number $\alpha(G)$ of a graph $G$ is the maximum cardinality of a set of pairwise nonadjacent vertices in $G$. In this paper, we study the independence number of the modular product of graphs. We first structurally characterize all independent set of $G\diamond H$ which lead to the exact result on  $\alpha(G \diamond H)$. Special cases of this result lead to several sharp bounds and some exact results for $\alpha(G \diamond H)$. Finally, we introduce a partition graph associated with $G \diamond H$ that provides a framework for constructing independent sets of the modular product from independent sets of its substructures.
\end{abstract}



\section{Introduction}

Taking a product of graphs provides a systematic way to construct a larger graph (the product) from smaller graphs (the factors). Conversely, decomposing a graph into its component factors often yields significant advantages. In particular, it is frequently possible to determine or estimate properties of the product graph from corresponding or related properties of the factor graphs. Moreover, such factorizations often facilitate algorithmic investigations.

The study of graph products can be divided into two main strands. The first concerns determining whether a graph $G$ can be expressed as a product at all; that is, whether there exist (smaller) graphs $G_1,\dots,G_k$ such that, for some graph product $*$, we have
\[
G = G_1 * \cdots * G_k.
\]
A closely related question is whether such a factorization is unique. These problems have been answered in the affirmative for several graph products and graph classes, and polynomial-time algorithms for such factorizations have been developed. The main focus has been on four standard graph products: the Cartesian, the strong, the direct, and the lexicographic products. For the Cartesian product, an optimal (i.e., linear-time) algorithm was presented in \cite{ImPe} after two decades of development in the area. Polynomial-time algorithms are known for factorizations with respect to the strong product \cite{FeSc} and the direct product \cite{Imri}. A different approach, based on the so-called Cartesian skeletons, was introduced in \cite{HaIm}. Factorization with respect to the lexicographic product is closely related to the graph isomorphism problem. For a detailed account of these results, we refer the reader to the comprehensive monograph on graph products \cite{HaIK}.

The second strand of inquiry studies the relationships between a product graph and its factors. This line of research has been immensely popular in the graph theory community over the past decades. A prominent example of this type is Hedetniemi's conjecture,
\[
\chi(G \times H) = \min\{\chi(G), \chi(H)\},
\]
concerning the chromatic number of the direct product; this conjecture was disproved only recently in \cite{shitov}. Another classical example is Vizing's conjecture,
\[
\gamma(G \Box H) \geq \gamma(G)\gamma(H),
\]
for the domination number of the Cartesian product; see the recent survey \cite{BDG-12} for a detailed discussion of this conjecture, which remains open. Both conjectures have stimulated a large body of deep and influential work.

Although there are $256$ different graph products whose edge sets are determined by the edge sets of their two factors, only $20$ of these products are associative; see \cite{Imri1,ImIz}. Moreover, only $10$ of the associative products are also commutative, among them the empty product and the complete product, which are regarded as trivial. The remaining products can be grouped into pairs, where each pair consists of a product $G * H$ and its complementary product
\[
G \overline{*} H \stackrel{\mathrm{def}}{=} \overline{\overline{G} * \overline{H}},
\]
with $\overline{G}$ denoting the complement of $G$. Three of these pairs correspond to the standard products mentioned above, namely the Cartesian, the direct, and the strong products, together with their complementary counterparts. (Note that the lexicographic product is not commutative.) This classification leaves a single remaining product, known as the \emph{modular product}. One might also refer to it as the ``forgotten'' associative and commutative product, since after the foundational work in \cite{Imri1} it was not studied for more than 40 years. Despite its natural definition and structural properties, the modular product has received surprisingly little attention in the literature. Beyond the foundational work on its factorization and basic properties \cite{Imri1}, only a small number of graph invariants have been investigated in this context. In particular, its connection to $L(2,1)$-labelings was studied in~\cite{shao-solis}. The domination number of the modular product were treated in~\cite{BPS-25}, and connections to metric-type parameters, such as the distance and the strong metric dimension, have also been explored~\cite{KK-2025}. Nevertheless, a systematic investigation of classical graph invariants for the modular product is still largely missing.

Among these invariants, the independence number stands out as a particularly natural and well-motivated object of study. Independence is a fundamental notion in graph theory and plays a central role in combinatorial optimization and algorithmic applications. Moreover, for several standard graph products, the independence number exhibits rich and often subtle interactions with the corresponding parameters of the factor graphs. Given the close relationship between the modular product and other associative and commutative products, it is therefore natural to ask how the independence number of the modular product relates to the independence numbers and structural properties of its factors. Addressing this question not only contributes to a deeper understanding of the modular product itself, but also complements and extends existing results on graph invariants under graph products.

In the next section, we introduce the necessary terminology and present several fundamental results that provide trivial bounds on the independence number of the modular product of two graphs.
In Section~\ref{sec:3}, we structurally characterize independent sets of modular product of two graphs using sequences of disjoint independent sets of $G$ and $H$, respectively. This characterization also implies an exact formula for $\alpha(G \diamond H)$. In Section~\ref{s:consequences} we use this structural characterization of independent sets in modular product and
derive lower bounds for the independence number of the modular product in terms of various graph parameters of its factors. In particular, we establish a bound expressed through the independence number of one factor and the 1-local independence number of the other; this bound is shown to be sharp for grid graphs.
Structural properties of independent sets in the modular product, are then employed to prove that $n(G) + n(H) - 2$ serves as an upper bound for the independence number of the modular product. The characterization of graphs achieving this bound is also presented.
Finally, in Section~\ref{sec:4}, we study a partition graph $GH$ associated with the modular product $G \diamond H$. This construction provides insight into how independent sets of the modular product can be obtained from independent sets of smaller components of the product.

%%%%%%%%%%%%%%%%%%%%%%%%%%%%%%%%%%%%%%%%%%%%%%%%

\section{Preliminaries}

Throughout this paper, we consider only simple, undirected graphs. Let $G$ be a graph. By $\overline{G}$ we denote the \emph{complement} of $G$, that is, the graph with vertex set $V(\overline{G}) = V(G)$ in which two vertices are adjacent if and only if they are not adjacent in $G$. The \emph{order} of a graph $G$, $n(G)$, is the cardinality of its vertex set and the \emph{size} of $G$ is the cardinality of $E(G)$. As usual, $K_n,P_n,C_n$ denotes a complete graph, a path, and a cycle, respectively, on $n$ vertices, while we use $N_n\cong\overline{K}_n$.  For a subset $S \subseteq V(G)$, the graph $G[S]$ denotes the subgraph of $G$ induced by $S$.

For a vertex $v$ of a graph $G$, the \emph{open neighborhood} of $v$, denoted by $N_G(v)$ (or simply $N(v)$ when the graph is clear from the context), is the set of all vertices of $G$ adjacent to $v$. The \emph{closed neighborhood} $N_G[v]$ of $v$ is defined as $N_G[v] = N_G(v) \cup \{v\}$. Moreover, the complement of the closed neighborhood of $v$ is defined by
\[
\overline{N}_G[v] = \{u \in V(G) \setminus \{v\} : uv \notin E(G)\}.
\]

The \emph{distance} between two vertices $u,v \in V(G)$ is the number of edges on a shortest $u,v$-path in $G$. For $i \geq 1$, the set of vertices at distance exactly $i$ from a vertex $v$ is denoted by $N_i(v)$.

A set $A \subseteq V(G)$ is an \emph{independent set} of $G$ if no two vertices of $A$ are adjacent in $G$. The maximum cardinality of an independent set in $G$ is called the \emph{independence number} of $G$ and is denoted by $\alpha(G)$. An independent set of cardinality $\alpha(G)$ is referred to as an $\alpha(G)$-set. 

As usual, we write $[k]$ for the set $\{1,\dots,k\}$. The \emph{chromatic number} of a graph $G$, denoted by $\chi(G)$, is the smallest integer $k$ for which there exists a mapping $f \colon V(G) \to [k]$ such that $f(u) \neq f(v)$ whenever $uv \in E(G)$. For such a mapping, let
\[
C_i = \{u \in V(G): f(u) = i\}, \quad i \in [k].
\]
Clearly, $C_1,\dots,C_k$ form a partition of $V(G)$ into independent sets.

A graph $G$ is \emph{$t$-partite} if its vertex set can be partitioned into $t$ non-empty, pairwise disjoint sets $V_1,\dots,V_t$ such that for any $i \in [t]$ the graph $G[V_i]$ is edgeless graph. A $t$-partite graph $G$ is said to be \emph{complete $t$-partite} if, for any $x,y \in V(G)$, we have $xy \in E(G)$ if and only if there exist distinct $i,j \in [t]$ such that $x \in V_i$ and $y \in V_j$. If $|V_i| = p_i$ for each $i \in [t]$, then the complete $t$-partite graph with partition $V_1,\dots,V_t$ is denoted by $K_{p_1,\dots,p_t}$. Throughout the paper, we assume without loss of generality that
\[
p_1 \geq p_2 \geq \cdots \geq p_t.
\]
For $t=2$, we obtain the complete bipartite graph $K_{p,r}$. Note that a complete $1$-partite graph is an edgeless graph, and that $K_n$ can also be viewed as the complete $n$-partite graph $K_{1,\dots,1}$.

Let $G$ and $H$ be two graphs. Different products between $G$ and $H$ have vertex set $V(G)\times V(H)$, but their edge sets depends on the product. 
By $G\diamond H$ we denote the \emph{modular product}. Two vertices $(g,h)$ and $(g',h')$ are adjacent in $G\diamond H$ if they are equal in one coordinate ($g=g'$ or $h=h'$) and are adjacent in the other coordinate ($hh'\in E(H)$ or $gg'\in E(G)$, respectively), or they are adjacent in both coordinates ($gg'\in E(G)$ and $hh'\in E(H)$), or (for $g\neq g'$ and $h\neq h'$) they are nonadjacent in both coordinates ($gg'\notin E(G)$ and $hh'\notin E(H)$), see the left side of  Figure \ref{fig1} for an example of a modular product. The edges of the first condition are called \emph{Cartesian edges} because they define the edges of the Cartesian product $G\Box H$. The edges arising from the second condition are called \emph{direct}, as they form the edges of the \emph{direct product} $G\times H$. The third condition gives \emph{co-direct} edges, since they are the edges of the direct product of complements of $G$ and $H$, that is $\overline{G}\times\overline{H}$. Hence,
$$E(G\diamond H)=E(G\Box H)\cup E(G\times H)\cup E(\overline{G}\times\overline{H}).$$
Recall that Cartesian and direct edges together define the edge set of the \emph{strong product} $G\boxtimes H$, another important graph product. For a vertex $h\in V(H)$, we call the set $G^{h}=\{(g,h)\in V(G \diamond H):g\in
V(G)\}$ a $G$-\emph{layer} of $G \diamond H$. By abuse of notation we will also consider $G^{h}$ as the corresponding induced subgraph. Clearly $G^{h}$ is isomorphic to $G$. For $g\in V(G)$, the $H$-\emph{layer} $^g\!H$ is defined as $^g\!H =\{(g,h)\in V(G \diamond H)\,:\,h\in V(H)\}$. We will again also consider $^g\!H$ as an induced subgraph and note that it is isomorphic to $H$. A mapping $p_{G}:V(G \diamond H)\rightarrow V(G)$, $p_{G}(g,h) = g$ is the \emph{projection} onto $G$ and $p_{H}:V(G \diamond H)\rightarrow V(H)$, $p_{H}(g,h) = h$ the \emph{projection} onto $H$. For a set $A$ of $G \diamond H$ and vertices $x \in V(G), y \in V(H)$, we denote by $^x\!A =\ ^x\!H \cap A$ and $A^y=G^y\cap A$.

The closed neighborhood of a vertex in the modular product, as seen immediately from the definition, is given by
\begin{equation}\label{modular_nbrs}
N_{G\diamond H}[(g,h)]=(N_G[g]\times N_H[h])\cup (\overline{N}_G[g]\times \overline{N}_H[h]).
\end{equation}

It follows also directly from the definition that $G \boxtimes H$ is a spanning subgraph of $G \diamond H$ and thus we get a trivial upper bound for the independence number of the modular product of two graphs, $\alpha(G \diamond H) \leq \alpha(G \boxtimes H)$. Note that the independence number of strong product graphs was intensively studied, so there are many known upper bounds and also exact values in the case when graphs $G$ and $H$ are from some special graph families~\cite{Hales, NR-96, SK-74, Vesel-1998}. Since any upper bound for $\alpha(G \boxtimes H)$ is also an upper bound for $\alpha(G \diamond H)$, we for example obtain $$\alpha(G \diamond H) \leq p(G)\alpha(H),$$ where $p(G)$ is the \emph{Rosenfeld number} of $G$~\cite{Hales}, or  
$$\alpha(G \diamond H) \leq \left\lfloor |V(G)|\cdot \frac{\alpha(H )}{\omega(G)} \right\rfloor,$$ holds for any vertex-transitive graph $G$, where $\omega(G)$ denotes the {\emph{ clique number}} of $G$ ~\cite{SK-74}. Moreover, since $ G \diamond K_n$ is isomorphic to $G \boxtimes K_n$, we immediately get that for any graph $G$ and any $n \in \mathbb{N}$,  
$$\alpha(G\diamond K_n) = \alpha(G).$$
Thus, in the rest of the paper we will be interested in computing the independence number of a modular product of two graphs where non of the factors is a complete graph.  

%%%%%%%%%%%%%%%%%%%%%%%%%%%%%%%%%%%%%%%%%%%%%%%%%%%%%%%%%%%%%%%%%%%%%%%%%%

\section{The structure of independent sets of modular product}\label{sec:3}

In this section we completely describe the structure of any independent set of modular product $G\diamond H$. For this let ${\cal A}=(A_1,\dots,A_t)$ and ${\cal B}=(B_1,\dots,B_t)$ be $t$-tuples of pairwise disjoint independent sets of $G$ and $H$, respectively, where $|A_i|=a_i$ and $|B_i|=b_i$ for every $i\in[t]$. Throughout the paper the notation from the previous sentence will have the same meaning if not specified otherwise. We say that ${\cal A}$ and ${\cal B}$ are \emph{friendly} if  
\begin{itemize}
\item $a_i=1$ or $b_i=1$ for every $i\in[t]$, and   
\item either $G[A_i\cup A_j]=K_{a_i,a_j}$ and $B_i\cup B_j$ is an independent set of $H$ or vice versa, $H[B_i\cup B_j]=K_{b_i,b_j}$ and $A_i\cup A_j$ is an independent set of $G$, for any different $i,j\in[t]$.
\end{itemize}
Observe that ${\cal A}=(\{g_1\},\{g_2\},\{g_3\},\{g_4\})$ and ${\cal B}=(\{h_1\},\{h_4\},\{h_2\},\{h_3,h_5\})$ are friendly sequences of independent sets of graphs $G$ and $H$ from Figure~\ref{fig1}.

Let $\cal{A}$ be a family of disjoint subsets of $V(G)$ of a graph $G$. The $G$-quotient graph obtained from $\cal{A}$ is the graph with vertex set $\cal{A}$, where two different vertices $A_i,A_j \in {\cal{A}}$ are adjacent if and only if for any $x \in A_i$ and any $y \in A_j$, $xy \in E(G)$. Directly from the definition of a friendly sequences  we obtain the following.

\begin{observation}\label{quotient}
    If ${\cal A}=(A_1,\dots,A_t)$ and ${\cal B}=(B_1,\dots,B_t)$ are friendly sequences of non-complete graphs $G$ and $H$, respectively, then the $G$-quotient graph obtained from $\cal{A}$ is isomorphic to the complement of the $H$-quotient graph obtained from $\cal{B}$.
\end{observation}

\begin{theorem}\label{friends}
   Let $G$ and $H$ be graphs and let $S\subseteq V(G\diamond H)$. A set $S$ is an independent set of $G\diamond H$ if and only if the pairwise distinct nonempty sets from $\{p_G(S^h):\, h\in V(H)\}$ and $\{p_H(^g\!S):\, g\in V(G)\}$ can be ordered as sequences $(A_1,\ldots ,A_k)$ and $(B_1,\ldots ,B_k)$, respectively, in such a way that the sequences are friendly and $S= \bigcup_{i=1}^k (A_i \times B_i)$.   
\end{theorem}

\begin{proof}
    Let $S$ be an independent set of $G\diamond H$. In the proof we will first construct sequences of sets from $\{p_G(S^h):h\in V(H)\}$ and $\{p_H(^g\!S):g\in V(G)\}$ and then prove that such sequences are friendly. We describe construction as an iterative procedure in the following way. Start with $X:=p_H(S)$ and $i:=1$. At step $i$ we do the following. Let $h$ be an arbitrary vertex in $X$. Let $A_i=p_G(S^h)$ and $a_i=|A_i|$. Since $S$ is an independent set of $G \diamond H$, $A_i$ is an independent set of $G$. If $a_i=1$, then there exists $g \in V(G)$ such that $A_i=\{g\}$. In this case define $B_i=p_H(^g\!S)$ and let $b_i=|B_i|$. Again note, that $B_i$ is an independent set of $H$. If $a_i > 1$, then we first show that for any $g \in A_i$, $p_H(^g\!S)=\{h\}$. Indeed, if there exists $h' \in p_H(^g\!S)\setminus \{h\}$, then for any $g' \in A_i \setminus \{g\}$ it holds that $(g,h),(g,h'),(g',h) \in S$. Since $S$ is an independent set of $G \diamond H$ $gg' \notin E(G), hh' \notin E(H)$. Therefore $(g,h')(g',h) \in E(G \diamond H)$, which contradicts the fact that $S$ is an independent set of $G \diamond H$. Hence for any $g \in A_i$, $p_H(^g\!S)=\{h\}$. Thus we can define, $B_i=\{h\}$ (or equivalently, choose arbitrary $g \in A_i$, then $B_i=p_H(^g\!S)$) and $b_i=1$. At the end of step $i$, we update $X:=X\setminus B_i$ and set $i:=i+1$. Now, if $X \neq \emptyset$, then we continue with step $i+1$ in such a way that we repeat the procedure at step $i$ with updated $X$. If $X=\emptyset$, then we stop the procedure. The procedure will clearly stop in a finite number of steps, since at each step the cardinality of $X$ is decreased.

    Say that procedure is finished after $p$ steps. We will show that the sequences $(A_1,\ldots , A_p)$ and $(B_1,\ldots , B_p)$ are friendly. Construction of $B_1,\ldots ,B_p$ directly implies that $\{B_1,\ldots , B_p\}$ is a partition of $p_H(S)$ to independent sets. Clearly, also sets in ${\cal{A}}=\{A_1,\ldots ,A_p\}$ are independent sets. Next, we show that $\cal{A}$ is a partition of $p_G(S)$. Suppose first that there exist $i < j \leq p$ such that $A_i \cap A_j \neq \emptyset$. Then there exists $x \in A_i \cap A_j$. By definition of sets from $\cal{A}$ (constructed in the above algorithm), it follows that there exist $h,h' \in V(H)$, $h \neq h'$ such that $A_i=p_G(S^h)$ and $A_j=p_G(S^{h'})$. If $a_i=a_j=1$ (and thus $A_i=A_j=\{x\}$), then the definition of $B_i$ implies that $h,h' \in B_i$. Thus $h,h'$ are contained in $X$ until the end of step $i$ and non of $h,h'$ is contained in $X$ at step $j$ of the above algorithm. Hence,  $A_j=p_G(S^{h'})$ is not possible. If $\max{\{a_i,a_j\}} > 1$, then assume without loss of generality that $a_i >1$. Then for any $g \in A_i \setminus \{x\}$ vertices $(g,h),(x,h),(x,h') \in S$. Since $S$ is independent set of $G \diamond H$, it follows that $gx \notin E(G)$ and $hh' \notin E(H)$. Hence $(g,h)(x,h') \in E(G \diamond H)$, a contradiction, and the sets from $\cal{A}$ are pairwise disjoint independent sets that clearly cover $p_G(S)$. 

    To finish this direction it remains to show that both conditions from the definition of friendly sequences are satisfied. From the construction of $A_i,B_i$ it follows that $a_i \neq 1$ implies $b_i=1$, which proves the first condition and implies that $S=\bigcup_{i=1}^p A_i \times B_i$. For the second condition we need to check three possibilities for any  $A_i,A_j\in{\cal A}$ and any $B_i,B_j\in{\cal B}$. First, if $G[A_i\cup A_j]\neq K_{a_i,a_j}$ and $B_i\cup B_j$ is an independent set of $H$, then $g_ig_j\notin E(G)$ for some $g_i\in A_i$ and $g_j\in A_j$. Hence, for any $h_i\in B_i$ and any $h_j\in B_j$, $(g_i,h_i), (g_j,h_j)$ are two vertices of $S$ that are adjacent in $G \diamond H$, a contradiction. By symmetry, we obtain the same contradiction if we exchange the role of $G$ and $H$. Next, if $G[A_i\cup A_j]= K_{a_i,a_j}$ and $B_i\cup B_j$ is not an independent set of $H$, then $h_ih_j\in E(H)$ for some $h_i\in B_i$ and $h_j\in B_j$. In this case, for any $g_i\in A_i$ and any $g_j\in A_j$ vertices $(g_i,h_i),(g_j,h_j)$ are from $S$ and are adjacent in $G \diamond H$, a contradiction. Again, we obtain the same contradiction if we exchange the role of $G$ and $H$. Finally, $G[A_i\cup A_j]\neq K_{a_i,a_j}$, $A_i\cup A_j$ is not independent in $G$, $H[B_i\cup B_j]\neq K_{b_i,b_j}$ and $B_i\cup B_j$ is not an independent set of $H$. Hence there exist $g,g' \in V(G[A_i\cup A_j])$ such that $gg' \in E(G)$ and there exist $h,h'\in V(H[B_i\cup B_j])$ such that $hh' \in E(H)$. This two edges now yield an edge between two vertices from $S$, a contradiction. Hence, the second condition is also fulfilled and thus $(A_1,\ldots ,A_p)$ and $(B_1,\ldots , B_p)$ are friendly sequences.

    Conversely, let $S \subseteq V(G \diamond H)$ and let ${\cal{A}}=(A_1,\ldots , A_t)$ be a sequence of distinct nonempty sets from $\{p_G(S^h):\, h\in V(H)\}$ and ${\cal{B}}=(B_1,\ldots , B_t)$ a sequence of distinct nonempty sets in $\{p_H(^g\!S):\, g\in V(G)\}$  such that sequences $(A_1,\ldots , A_t)$ and $(B_1,\ldots ,B_t)$ are friendly and $S=\bigcup_{i=1}^t(A_i \times B_i)$. We will show that $S$ is an independent set of $G \diamond H$. Note first that by the definition of friendly sequences, sets $A_1,\ldots ,A_t$ must be pairwise disjoint and independent in $G$ and sets $B_1,\ldots ,B_t$ must be pairwise disjoint and independent in $H$. As $S=\bigcup_{i=1}^t(A_i \times B_i)$ and for any $i \in [t]$ $a_i=1$ or $b_i=1$, it follows that  $(G\diamond H)[A_i \times B_i]$ is  isomorphic to $G[A_i]$ or $H[B_i]$ and thus it is an edgeless graph.  Hence there are no Cartesian edges between vertices of $S$. Fix any different $i,j\in[t]$. The second condition of friendliness implies that either we have all edges between vertices of $A_i$ and $A_j$ and no edge between $B_i$ and $B_j$ or no edges between vertices of $A_i$ and $A_j$ and all edges between $B_i$ and $B_j$. In both cases we have no edge between vertices of $A_i\times B_i$ and $A_j\times B_j$. Since $i$ and $j$ were arbitrary, $S$ must be an independent set of $G\diamond H$.
\end{proof}

By Theorem \ref{friends} and its proof, every independent set $S$ of $G\diamond H$ is generated by friendly sequences $(A_1,\dots,A_t)$ and $(B_1,\dots,B_t)$, more precisely $S=\bigcup_{i=1}^t A_i \times B_i$. Moreover, Theorem \ref{friends} implies that an independent set $S$ of $G \diamond H$ has the form presented on Figure~\ref{fig:independent}. Such is also an $\alpha(G\diamond H)$-set which implies the following.

\begin{corollary}\label{alpha}
Let $G$ and $H$ be graphs. Then
\[
\alpha(G \diamond H)
=
\max \left\{
\sum_{i=1}^{t} \bigl(|A_i|+|B_i|\bigr)-t
:
(\mathcal A,\mathcal B)\text{ are friendly and }
|\mathcal A|=|\mathcal B|=t
\right\}.
\]

%If $\cal{A}$ is a sequence of disjoint independent subsets of $V(G)$, $\cal{B}$ a sequence of disjoint independent subsets of $V(H)$ and $|{\cal{A}}|=|{\cal{B}}|=t$, then     
%$$\alpha(G \diamond H)=\max_{({\cal A},{\cal B}) \text{ are friendly}} \Big\{ \sum_{i=1}^t(|A_i|+|B_i|) - t  \Big\}.$$ 
\end{corollary}

\begin{figure}[ht!]
\begin{center}
\begin{tikzpicture}[scale=1,style=thick,x=1cm,y=1cm]
\def\vr{4pt} % \vr = vertex radius;

% define vertices
%%%%%
%%%%%
\path (0,0) coordinate (a1);
\path (0.5,0) coordinate (a2);
\path (1.5,0) coordinate (a3);
\draw (1,0) node {$\ldots$};

\path (2.2,0.7) coordinate (b1);
\path (2.7,0.7) coordinate (b2);
\path (3.7,0.7) coordinate (b3);
\draw (3.2,0.7) node {$\ldots$};

\path (4.4,1.4) coordinate (c1);
\path (4.4,1.9) coordinate (c2);
\path (4.4,2.9) coordinate (c3);
\draw (4.4,2.4) node {$\vdots$};

\path (5.1,3.6) coordinate (d1);
\path (5.6,3.6) coordinate (d2);
\path (6.6,3.6) coordinate (d3);
\draw (6.1,3.6) node {$\ldots$};

\path (7.3,4.3) coordinate (e1);
\path (7.3,4.8) coordinate (e2);
\path (7.3,5.8) coordinate (e3);
\draw (7.3,5.3) node {$\vdots$};

\path (0,-1.5) coordinate (y1);
\draw (y1) circle (\vr);
\path (1.5,-1.5) coordinate (y2);
\draw (y2) circle (\vr);
\draw (0.75,-1.5) node {$\ldots$};
\draw (0.7,-1.5) ellipse (1.1cm and 0.3cm);
\draw (0.75,-2.2) node {$A_1$};

\path (2.2,-1.5) coordinate (y3);
\draw (y3) circle (\vr);
\path (3.7,-1.5) coordinate (y4);
\draw (y4) circle (\vr);
\draw (2.9,-1.5) node {$\ldots$};
\draw (3,-1.5) ellipse (1.1cm and 0.3cm);
\draw (2.95,-2.2) node {$A_2$};

\path (4.4,-1.5) coordinate (y5);
\draw (y5) circle (\vr);
\draw (4.4,-2.2) node {$A_3$};

\path (5.1,-1.5) coordinate (y6);
\draw (y6) circle (\vr);
\path (6.6,-1.5) coordinate (y7);
\draw (y7) circle (\vr);
\draw (5.8,-1.5) node {$\ldots$};
\draw (5.85,-1.5) ellipse (1.1cm and 0.3cm);
\draw (5.85,-2.2) node {$A_4$};

\path (7.3,-1.5) coordinate (y8);
\draw (y8) circle (\vr);
\draw (7.3,-2.2) node {$A_5$};

\path (-1.5,0) coordinate (x1);
\draw (x1) circle (\vr);
\draw (-2.1,0) node {$B_1$};

\path (-1.5,0.7) coordinate (x2);
\draw (x2) circle (\vr);
\draw (-2.1,0.7) node {$B_2$};
\path (-1.5,1.4) coordinate (x3);
\draw (x3) circle (\vr);
\path (-1.5,2.9) coordinate (x4);
\draw (x4) circle (\vr);
\draw (-1.5,2.25) node {$\vdots$};

\path (-1.5,3.6) coordinate (x5);
\draw (x5) circle (\vr);
\draw (-2.1,3.6) node {$B_4$};

\path (-1.5,4.3) coordinate (x6);
\draw (x6) circle (\vr);
\path (-1.5,5.8) coordinate (x7);
\draw (x7) circle (\vr);
\draw (-1.5,5.1) node {$\vdots$};

\draw (-1.5,2.2) ellipse (0.3cm and 1.1cm);
\draw (-2.1,2.2) node {$B_3$};
\draw (-1.5,5.1) ellipse (0.3cm and 1.1cm);
\draw (-2.1,5.1) node {$B_5$};

\draw (-1,-1)--(10,-1)--(10,8)--(-1,8)--(-1,-1);

\draw (a1) [fill=black] circle (\vr); 
\draw (a2) [fill=black] circle (\vr); 
\draw (a3) [fill=black] circle (\vr); 
\draw (b1) [fill=black] circle (\vr); 
\draw (b2) [fill=black] circle (\vr); 
\draw (b3) [fill=black] circle (\vr); 
\draw (c1) [fill=black] circle (\vr); 
\draw (c2) [fill=black] circle (\vr); 
\draw (c3) [fill=black] circle (\vr); 
\draw (d1) [fill=black] circle (\vr); 
\draw (d2) [fill=black] circle (\vr); 
\draw (d3) [fill=black] circle (\vr); 
\draw (e1) [fill=black] circle (\vr); 
\draw (e2) [fill=black] circle (\vr); 
\draw (e3) [fill=black] circle (\vr); 

\end{tikzpicture}
\end{center}
\caption{Independent set in modular product} \label{fig:independent}
\end{figure}
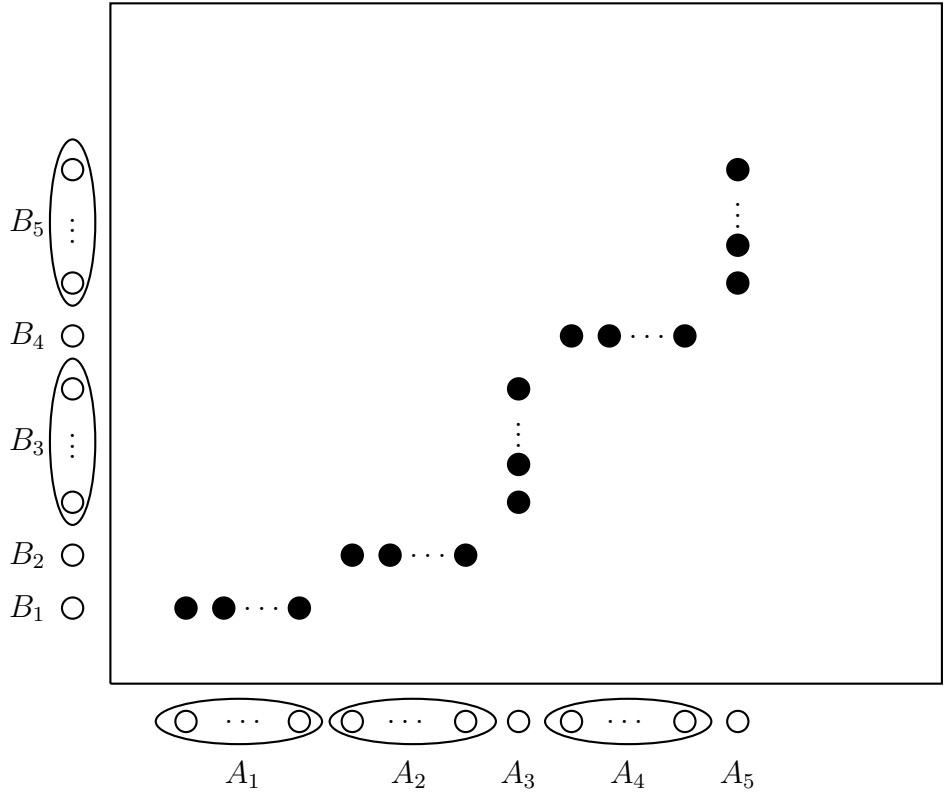

%%%%%%%%%%%%%%%%%%%%%%%%%%%%%%%%%%%%%%%%%%%%%%%%%%%%%%%%5
\section{Some consequences of Theorem \ref{friends}}\label{s:consequences}

In this section we present some special cases of friendly sequences ${\cal A}$ and ${\cal B}$ of $G$ and $H$, respectively, that yield several sharp bounds of $\alpha(G \diamond H)$. We start with $t=1$, that is ${\cal A}=(A_1)$ and ${\cal B}=(B_1)$. We may have the following two options: $a_1\geq 1$ and $b_1=1$ or $a_1=1$ and $b_1\geq 1$. In the first case we also have $a_1\leq \alpha(G)$ and in the second $b_1\leq \alpha (H)$. 

\begin{proposition}\label{alphaBound}
    For any graphs $G$ and $H$ of order at least three we have 
    $$\alpha(G \diamond H) \geq \max\{\alpha(G),\alpha(H)\}$$ 
    and the bound is sharp. 
\end{proposition}

\begin{proof}
    The bound follows from Corollary~\ref{alpha} for $t=1$. The bound is clearly sharp, since for any positive integer $n$ and any graph $G$, $\alpha(G \diamond K_n)=\alpha(G) =\max{\{\alpha(G),\alpha(K_n)\}}$.
\end{proof}

As proved in previous proposition, if one graph $G$ or $H$ is a complete graph, then the bound from Proposition~\ref{alphaBound} is sharp. Anyway, the bound can be sharp also if both graphs $G$ and $H$ are non-complete. For example, let $G$ be a disjoint union of $K_1$ and $K_2$ and let $H=N_2$. Then $G \diamond H$ is isomorphic to $2K_3$ and thus $\alpha(G \diamond H)=2 = \max{\{\alpha(G), \alpha(H)\}}$. This leads to the problem of determining all pairs of graphs $G,H$ for which $\alpha(G \diamond H) = \max\{\alpha(G),\alpha(H)\}$. We will present a necessary condition for this problem later, after the introduction of the local independence number.

An independent set treated in Proposition \ref{alphaBound} lie completely in one $G$-layer or one $H$-layer and we have $b_i=1$ or $a_i=1$, respectively. A generalization of this idea is that we have $b_i=1$ for each $i\in [t]$ (or symmetrically $a_i=1$ for each $i\in[t]$). Next we deal with this special case of ${\cal A}$ and ${\cal B}$ that are friendly. For this let $\cp(G)$ denote the order of the maximum induced complete multipartite subgraph, i.e.\ $$\cp(G)=\max{\{|S|;G[S] \textrm{ is a complete multipartite graph}\}}.$$
Note that $\cp(G)=\max{\{|S|; \overline{G}[S] \textrm{ is a disjoint union of cliques}\}}$. This invariant (the order of a largest subgraph that is a disjoint union of cliques) was already studied~\cite{HB-2021, FG-10} and is closely related to another known graph invariant called \emph{cluster vertex deletion number}~\cite{DK-2012, GG-2004}. The problem was not studied just for the disjoint union of cliques but also for other types of induced subgraphs. Since edgeless graph is a complete 1-partite graph, it is clear that $\cp(G) \geq \alpha(G)$.

Nevertheless, we introduce a variation of $\cp(G)$ that avoids complete multipartite subgraphs having too many partition sets. Thus, for a fix positive integer $r\geq 2$ and a graph $G$ we define $\cp_{r}(G)$ as the maximum number of vertices of a complete $k$-partite subgraph of $G$ where $k\leq r$. It follows directly from the definition that $\cp_{r}(G)\leq\cp(G)$. We have $\cp_{r}(K_n)=\min\{n,r\}$, for $n\geq 2$, as $K_n=K_{\underbrace{1,\ldots ,1}_{n}}$ and thus $K_n$ contains complete $n$-partite graph of order $n$.  

\begin{theorem}\label{setA}
If $G$ and $H$ are non-complete graphs, then   
$$\alpha(G\diamond H)\geq \max\{\cp_{\alpha(H)}(G),\cp_{\alpha(G)}(H)\}$$
and the bound is sharp for $K_{p_1,\dots,p_t}\diamond N_t$ for $t\geq 2$. 
\end{theorem}

\begin{proof}
Let $K_{a_1,\ldots , a_t}$ be a $t$-partite induced subgraph of $G$ with $2 \leq t \leq \alpha(H)$ such that $a_1+\cdots +a_t=\cp_{\alpha(H)}(G)$. Moreover let $\{A_1,\dots,A_t\}$ be the $t$-partition of $K_{a_1,\ldots , a_t}$. Let $\{h_1,\ldots , h_t\}$ be an independent set of $H$. Then $(A_1,\ldots ,A_t)$ and $(\{h_1\},\ldots , \{h_t\})$ are clearly friendly and thus by Theorem~\ref{friends}, $S=\cup_{i=1}^t(A_i\times\{h_i\})$ is an independent set of $G\diamond H$  of cardinality $\cp_{\alpha(H)}(G)$. By symmetric argument we get an independent set of cardinality $\cp_{\alpha(G)}(H)$ and the bound follows. 

To prove that the bound is sharp for $K_{p_1,\ldots , p_t} \diamond N_t$, we have to show that $\alpha(K_{p_1,\ldots , p_t} \diamond N_t) \leq \cp_{t}(K_{p_1,\ldots , p_t})=p_1+\cdots +p_t$, since the other inequality follows from already proved bound. We will use induction on $t$ and show more general bound, that for any $\ell \leq t$, $\alpha(K_{p_1,\ldots , p_t} \diamond N_\ell) \leq p_1+\cdots +p_\ell$. For the base case, let $G$ be one partite graph, i.e.\ it is an edgeless graph, say $G=N_p$. The base case is completed vecause $\alpha(N_p \diamond N_1)=\alpha(N_p)=p=\cp_{\alpha(H)}(N_p)=\max{\{\cp_{\alpha(N_1)}(N_p),\cp_{\alpha(N_p)}(N_1)\}}$.

Let S be an $\alpha(K_{p_1,\ldots , p_t} \diamond N_\ell)$-set  for $\ell \leq t$. Denote $V(N_\ell)=\{y_1,\ldots ,y_\ell\}$ and let $V_1,\ldots ,V_t$ be a partition of $V(K_{p_1,\ldots ,p_t})$ with $|V_i|=p_i$ for any $i \in [t]$.  Let $x$ be an arbitrary vertex of $V(K_{p_1,\ldots ,p_t})$ such that $^x\!S \neq \emptyset$. Without loss of generality, let $x \in V_1$
 and let $\{y_1,\ldots ,y_i\}=p_{N_\ell}(^x\!S)$. Assume first that $i\geq 2$. Since $S$ is an independent set $(x,y_1),\ldots ,(x,y_i)$ are the only vertices of $S$ in the subgraph of $K_{p_1,\ldots ,p_t} \diamond N_\ell$ induced by $X=(V_1 \times V(N_\ell)) \cup (V(K_{p_2,\ldots ,p_\ell}) \times \{y_1,\ldots ,y_i\})$. Since the subgraph of $K_{p_1,\ldots ,p_t} \diamond N_\ell$ induced by the complement of $X$ (i.e.\ by $V(K_{p_1,\ldots ,p_t} \diamond N_\ell) \setminus X$) is isomorphic to $K_{p_2,\ldots , p_t} \diamond N_{\ell-i}$, the induction assumption implies that $|S| \leq i+ \alpha(K_{p_2,\ldots ,p_t} \diamond N_{\ell-i}) \leq i+p_2+\cdots +p_{\ell-i+1} \leq p_1+\cdots +p_\ell$ as $p_1+p_{\ell-i+2}+\cdots +p_\ell \geq i$ because $p_k \geq 1$ for any $k\in[t]$. 

We are left with $i=1$. Let $(A_1,\ldots ,A_t)$ and $(B_1,\ldots ,B_t)$ be a friendly partition for $S$ according to Theorem \ref{friends} such that $x\in A_1$. We have $b_1=1$ because $i=1$ and $a_1\leq p_1$ follows. Again, $A_1\times B_1$ are the only vertices of $S$ in the subgraph of $K_{p_1,\ldots ,p_t} \diamond N_\ell$ induced by $X=(V_1 \times V(N_\ell)) \cup (V(K_{p_2,\ldots ,p_\ell}) \times \{y_1\})$. As before the subgraph of $K_{p_1,\ldots ,p_t} \diamond N_\ell$ induced by the complement of $X$ is isomorphic to $K_{p_2,\ldots , p_t} \diamond N_{\ell-1}$. So,  the induction assumption implies that $|S| \leq a_1+ \alpha(K_{p_2,\ldots ,p_t} \diamond N_{\ell-1}) \leq p_1+p_2+\cdots +p_{\ell}$.
\end{proof}

Above lower bound is sharp for the two above mentioned families of graphs and we suspect that much more share this property. So, the next problem seems natural. 

\begin{problem}\label{problem1}
  Describe all pairs of graphs $G$ and $H$ with 
  $$\alpha(G \diamond H)=\max\{\cp_{\alpha(H)}(G),\cp_{\alpha(G)}(H)\}.$$ 
\end{problem}

We continue with $t=2$ where ${\cal A}=(A_1,A_2)$ and ${\cal B}=(B_1,B_2)$ are friendly. Now, we have two non-symmetric cases, which are $a_1\geq a_2\geq 1$ and $b_1=b_2=1$ or  $a_1\geq 2$, $a_2=b_1=1$ and $b_2\geq 2$. First one is a special case of Theorem \ref{setA} and we proceed with the second one. For this we recall some already investigated graph parameter. The {\emph{local independence number at distance}} $i$ of a graph $G$, $\alpha_i(G)$, is the maximum number of independent vertices at distance $i$ from any vertex, i.e.\ $\alpha_i(G)=\max{\{\alpha(G[N_i(x)]); x \in V(G)\}}$~\cite{FRS-4}. To present a lower bound for the independence number of a modular product we need local independence number at distance $1$. Note that if $\alpha_1(G)=k$, then $K_{k,1}$ is the largest induced star in $G$.   

\begin{theorem}\label{thm:lower1}
    If $G$ and $H$ are arbitrary non-complete graphs, then $$\alpha(G \diamond H) \geq \max{\{\alpha(G)+\alpha_1(H)-1, \alpha(H)+\alpha_1(G)-1\}}$$
    and the bound is sharp for $P_q\diamond P_r$ for $q\geq r\geq 5$, that is $\alpha(P_q\diamond P_r)=\lceil \frac{q}{2}\rceil+1$.
\end{theorem}
\begin{proof}
    Let $A=\{g_1,\ldots , g_k\}$ be an $\alpha(G)$-set and let $B_2=\{h_1,\ldots ,h_\ell\}$ be an independent set of $H$ that is contained in $N_H(h)$ for some $h \in V(H)$ and $|B_2|=\alpha_1(H)$. Define $A_1=A-\{g_1\}, A_2=\{g_1\}$ and $B_1=\{h\}$. Then it is straightforward to check that $(A_1,A_2)$ and $(B_1,B_2)$ are friendly sequences of $G$ and $H$, respectively. By Theorem~\ref{friends} the set   
    $$S=\{(g_1,h_1), (g_1,h_2),\ldots ,(g_1,h_\ell), (g_2,h),\ldots ,(g_k,h)\}$$ is an independent set of $G \diamond H$ of cardinality $\alpha_1(H)+\alpha(G)-1$. By symmetric argument the bound follows.

    For $P_q\diamond P_r$ recall that $\alpha(P_q)=\lceil \frac{q}{2}\rceil$ and $\alpha_1(P_r)=2$. Hence, $\alpha(P_q\diamond P_r)\geq \lceil \frac{q}{2}\rceil+1$ by our bound. Let ${\cal A}=(A_1,\dots,A_t)$ and ${\cal B}=(B_1,\dots, B_t)$ be friendly sequences of independent sets of $P_q$ and $P_r$, respectively, that yield an $\alpha(P_q\diamond P_r)$-set $S$ by Corollary~\ref{alpha}. If $t=1$, then for $b_1=1$ and $1\leq a_1\leq \alpha(P_q)=\lceil \frac{q}{2}\rceil$ we have $\alpha(P_q\diamond P_r)=|S|\leq\lceil \frac{q}{2}\rceil$, a contradiction. If $a_1=1$, then the same contradiction follows because we get $\alpha(P_q\diamond P_r)=|S|\leq\lceil \frac{r}{2}\rceil\leq\lceil \frac{q}{2}\rceil$. 

    If $t=2$, then one projection of $S$ (which is either $A_1 \cup A_2$ or $B_1 \cup B_2$) yields a complete bipartite graph $K_{2,1}$ or $K_{1,1}$ because $K_{2,1}$, $K_{1,1}$ and $K_1$ are the only possible complete multipartite subgraphs of the path. The other projection is an independent set of a path. In the case of $K_{1,1}$ we get a contradiction because $|S|\leq\lceil \frac{q}{2}\rceil$ in this case. For $K_{2,1}$ we obtain the set described in first paragraph of this proof and we have $\alpha(P_q\diamond P_r)=|S|\leq\lceil \frac{q}{2}\rceil+1$. 

    Let now $t\geq 3$ and let the sets $(A_1,\ldots,A_t)$ be ordered in non-increasing order of their cardinality. Assume first that $a_1>1$ and $a_2>1$. Thus $P_q[A_1\cup A_2]$ is not a complete bipartite subgraph because $K_{2,1}$ is a maximum complete bipartite subgraph of $P_q$. Hence, $A_1\cup A_2$ is an independent set of $G$ since ${\cal A}$ and ${\cal B}$ are friendly and consequently  $b\in B_1$ is adjacent to $b' \in B_2$ (note that $b_1=b_2=1$ as $a_1,a_2 > 1$). Let $i \in [t] \setminus \{1,2\}$. If $A_1 \cup A_i$ and $A_2 \cup A_i$ are independent sets, then  $B_1 \cup B_i$ and $B_2 \cup B_i$ induces a complete bipartite graph. Hence vertices $b \in B_1, b'\in B_2, b'' \in B_i$ induces $K_3$ in $P_r$, a contradiction. Thus, for any $i \in [t] \setminus \{1,2\}$, $A_1 \cup A_i$ or $A_2 \cup A_i$ induces a complete bipartite graph in $P_q$. Without loss of generality let $G[A_1 \cup A_3]$ be a complete bipartite graph. Then $a_1=2$ and $a_3=1$ and the vertex $x$ from $A_3$ is a common neighbor of both vertices from $A_1$. Since $x$ has no other neighbors in $P_q$, $A_2 \cup A_3$ is an independent set. From this we can easily deduce that $b_3=1$ and the $P_r$-quotient graph obtained from $\cal{B}$ on $B_1,B_2$ and $B_3$ is $P_3$ (in the same order $B_1B_2B_3$). Suppose that $t \geq 4$. Then since $A_1 \cup A_4$ is independent set of $P_q$, it follows that $A_2 \cup A_4$ must induce a complete bipartite graph. But then $a_4=1$ and $y \in A_4$ is a common neighbor of both vertices from $A_2$. Hence $A_1 \cup A_4$ and $A_3 \cup A_4$ are independent and thus since $\cal{A}$ and $\cal{B}$ are friendly, $B_4$ is adjacent to $B_1$ and $B_3$ in $H$-quotient graph obtained from $\cal{B}$. Moreover, $B_4$ is not adjacent to $B_2$ in this graph. Hence if we take one vertex from each of $B_1,\ldots ,B_4$, these four vertices induces $C_4$ in $P_q$, a contradiction. Hence $t=3$. Since $b_1=b_2=b_3=1$, $|S|=a_1+a_2+a_3$. Since $A_1 \cup A_2$ is an independent set of $P_q$ and $a_3=1$, we get $|S|\leq \alpha(P_q)+1=\lceil \frac{q}{2}\rceil+1$.
    
     Finally, assume that $a_1\geq 1$ and $a_i=1$ for any $i \in [t] \setminus \{1\}$. By symmetric argument we get, say $b_2\geq 1$ and $b_j=1$ for any $j \in [t]\setminus \{2\}$. Moreover, $a_1\leq 2$ or $b_2\leq 2$ again because $K_{1,2}$ is a maximum complete bipartite subgraph of a path and ${\cal A}$ and ${\cal B}$ are friendly. Thus let $A_i=\{x_i\}$ for any $i \in [t]\setminus \{1\}$ and let $B_i=\{y_i\}$ for any $i \in [t] \setminus \{2\}$. First note that since $P_q$-quotient graph obtained from $\cal{A}$ is isomorphic to the complement of $P_r$-quotient graph obtained from $\cal{B}$ (see Observation~\ref{quotient}), $t \leq 4$. If $a_i=1,b_i=1$ for any $i \in [t]$, then $|S|=t\leq 4 \leq \lceil \frac{q}{2}\rceil +1$ as $q \geq 5$. Hence we may without loss of generality assume that $a_1 \geq 2$. Now we distinguish two cases with respect to $a_1$. First assume that $a_1 \geq 3$. Since we know that at least one of $a_1,b_2$ is bounded by 2, we get that $b_2 \leq 2$. Since for any $i \geq 2$, $P_q[A_1 \cup A_i]$ cannot be a complete bipartite graph (note that the largest complete bipartite graphs in paths have order 3), it follows that $P_q[A_1 \cup A_i]$ is an independent set and since $\cal{A}$, $\cal{B}$ are friendly, $P_r[B_1 \cup B_i]$ is a complete bipartite graph $K_{1,1}$ or $K_{2,1}$. Suppose first that $b_2 =2$ (we already know that $b_i=1$ for $i \in [t]\setminus \{2\}$) which implies that $P_r[B_1 \cup B_2]$ is isomorphic to $K_{2,1}$. Denote $B_2=\{h,h'\}$. Then $B_1 \cup B_2=\{y_1,h,h'\}$ induces a star with center $y_1$. Now, since $A_1 \cup A_3$ is an independent set in $G$, it follows that $B_1 \cup B_3$ induces a complete bipartite graph $K_{1,1}$. Hence $y_1$ has three neighbors $h,h',y_3$ in $P_r$, a contradiction. Thus, $b_2=1$ and for any $i \in [t]\setminus \{1\}$ it holds that $B_1 \cup B_i$ induces a complete bipartite graph $K_{1,1}$. As the degree of $y_1$ in $P_r$ is at most 2, it follows that $t \leq 3$. Therefore $t=3$ because we are in case when $t \geq 3$. Since $b_1=b_2=b_3=1$, it follows that $|S|=|A_1|+|A_2|+|A_3|$. Moreover $A_1 \cup A_2$ is independent in $P_q$ and thus $|S|=|A_1|+|A_2|+a_3 \leq \alpha(P_q)+1=\lceil \frac{q}{2}\rceil+1$. In the same way we consider the case when $b_2 \geq 3$.

      At the end assume that $a_1=2$ and $b_2 \leq 2$. Denote $A_1=\{g,g'\}$. Assume first that there exists $i \in [t]\setminus \{1\}$ such that $P_q[A_1 \cup A_i]$ induces a complete bipartite graph (clearly $K_{2,1}$). Hence $B_1 \cup B_i$ is independent set. It follows that $A_1 \cup A_j$ and $A_i \cup A_j$ are independent sets for any $j \in [t] \setminus \{1,i\}$, because every vertex of $P_q$ has degree at most two. Since $\cal{A}$ and $\cal{B}$ are friendly, $B_1 \cup B_j$ and $B_i \cup B_j$ induces complete bipartite graphs in $P_r$. Since the vertices of $P_r$ have degrees at most 2, it follows that $b_i=b_j=1$ and $y_1,y_j,y_i$ induces a $P_3$ in $P_r$. If $t >3$, then the above argument holds for any $j_1,j_2 \neq 1,i$. Hence $y_1,y_{j_1},y_i$ and $y_1,y_{j_2},y_i$ induces a $P_3$ in $P_r$, which is clearly not possible. Hence $t=3$ and since $b_1=b_2=b_3=1=a_2=a_3$ and $a_1=2$, we get $|S|=a_1+a_2+a_3=4\leq \lceil \frac{q}{2}\rceil+1$. Finally, for any $i \in [t]\setminus \{1\}$ it holds that $A_1 \cup A_i$ is an independent set. Hence  $B_1 \cup B_i$ induces a complete multipartite graph $K_{1,1}$ or $K_{2,1}$ as $\cal{A}$ and $\cal{B}$ are friendly. Since $y_1$ has degree at most 2 in $P_r$ and since $t \geq 3$, we get that $t=3$ and $b_2=1$. Hence $|S|=a_1+a_2+a_3 \leq 4 \leq \lceil \frac{q}{2}\rceil+1.$     
\end{proof}

Note that the bound of Theorem~\ref{thm:lower1} is not sharp for $q=r=4$ because $\alpha(P_4 \diamond P_4)=4 \neq \lceil \frac{q}{2}\rceil +1 =3$, as $\{(v_1,v_2),(v_2,v_4),(v_3,v_1),(v_4,v_3)\}$ is an independent set of $P_4 \diamond P_4$ if $P_4=v_1v_2v_3v_4$. Despite this example, it seems that the lower bound of Theorem \ref{thm:lower1} is sharp for many modular products, whence the next problem.

\begin{problem}\label{problem2}
  Describe all pairs of graphs $G$ and $H$ with 
  $$\alpha(G \diamond H)=\max{\{\alpha(G)+\alpha_1(H)-1, \alpha(H)+\alpha_1(G)-1\}}.$$ 
\end{problem}

Now we return back to pairs of graphs $G,H$ for which $\alpha(G \diamond H)= \max{\{\alpha(G), \alpha(H)\}}$. Without loss of generality we may assume that graphs $G$ and $H$ are chosen such that $\alpha(G) \geq \alpha(H)$. 

\begin{proposition}\label{p:BoundCharacterization}
  Let $G$ and $H$ be graphs of order at least two with $\alpha(G) \geq \alpha(H)$. If $\alpha(G \diamond H)=\alpha(G)$, then $H$ is isomorphic to the disjoint union of at most $\alpha(G)$ complete graphs. 
\end{proposition}

\begin{proof}
  Let $G$ and $H$ be graphs of order at least two with $\alpha(G) \geq \alpha(H)$ and let $\alpha(G \diamond H)=\alpha(G)$. By Theorem~\ref{thm:lower1} we get $\alpha(G) =\alpha(G \diamond H) \geq \alpha(G)-1+\alpha_1(H)$. Hence $\alpha_1(H) \leq 1$. If $\alpha_1(H)=0$, then $H$ is isomorphic to edgeless graph $N_{n(H)}$ and hence it is a disjoint union of $n(H)$ complete graphs $K_1$. If $\alpha_1(H)=1$, then for any $h \in V(H)$ it holds that $N[h]$ induces a clique. Hence any component of $H$ is a complete graph, which completes the proof.
\end{proof}

It is clear that necessary condition from Proposition~\ref{p:BoundCharacterization} is not always sufficient. Already sharp examples from Theorem \ref{setA} provide us with some examples. Note that $N_2$ is union of complete graphs and by Theorem \ref{setA} we have $$\alpha(K_{q,r}\diamond N_2)=q+r> \max{\{\alpha(K_{q,r}) +\alpha_1(N_2)-1, \alpha(N_2)+\alpha_1(K_{q,r})-1\}}=$$ $$=\max{\{q+0-1,2+q-1\}}=q+1,$$ whenever $r>1$. 
Thus the next problem seems meaningful.  

\begin{problem}\label{problem3}
  Let $G$ be a graph of order at least two, and let $H$ be a disjoint union of $k$ complete graphs, where $k \leq \alpha(G)$. Determine all values of $k$ for which $\alpha(G \diamond H)=\alpha(G)$.
\end{problem}

%\begin{problem}\label{problem3}
%  Describe all pairs of graphs $G$ and $H$ with $\alpha(G \diamond H)=\alpha(G)$. 
%\end{problem}

Now we return back to the structure of $\alpha(G \diamond H)$-sets obtained from Theorem~\ref{friends} as depicted on Figure~\ref{fig:independent}. To proceed, we recall the formula for computing the independence number of the modular product shown in Corollary~\ref{alpha}, i.e.\  

\[
\alpha(G \diamond H)
=
\max \left\{
\sum_{i=1}^{t} \bigl(|A_i|+|B_i|\bigr)-t
:
(\mathcal A,\mathcal B)\text{ are friendly and }
|\mathcal A|=|\mathcal B|=t
\right\}.
\]

Behind every friendly pair ${\cal A}$ and ${\cal B}$ we have some disjoint independent sets $\{A_1,\dots,A_k\}$ of $G$ (and $\{B_1,\dots,B_k\}$ of $H$). If this independent sets further partition $V(G)$, then we obtain a coloring of $G$. This is a framework for the next upper bound.

\begin{proposition}\label{c:bound2}
    For graphs $G$ and $H$ we have      
$$\alpha(G \diamond H) \leq n(G)+n(H)-\max\{\chi(G),\chi(H)\}.$$
\end{proposition}

\begin{proof}
By Corollary~\ref{alpha}, there exist friendly sequences $(A_1,\ldots ,A_t)$ and $(B_1,\ldots ,B_t)$ of disjoint independent sets of $G$ and $H$, respectively, such that $\alpha(G \diamond H)=\sum_{i=1}^t(|A_i|+|B_i|)-t$ and $S=\bigcup_{i=1}^t (A_i \times B_i)$ being an $\alpha(G\diamond H)$-set. Denote $\sum_{i=1}^t|A_i|=n(G)-\ell$. Since $A_1,\ldots , A_t$ are disjoint independent sets of $G$, these sets correspond to coloring classes of $G[A_1\cup \ldots \cup A_t]$. To obtain proper coloring of whole graph $G$, we can color the remaining $\ell$ vertices of $G$, each with its own color. Hence $\chi(G) \leq t+\ell$ or equivalently $-t \leq -\chi(G)+\ell$. Thus, $$\alpha(G \diamond H)=\sum_{i=1}^t(|A_i|+|B_i|)-t=\sum_{i=1}^t|A_i|+\sum_{i=1}^t|B_i|-t$$
$$\leq n(G)-\ell +n(H) -\chi(G)+\ell =n(G)+n(H)-\chi(G).$$  
By symmetric arguments we obtain $\alpha(G \diamond H) \leq n(G)+n(H)-\chi(H)$ and the bound follows.
\end{proof}

Observe that above bound can be improved if $\chi(G)=\chi(H)=1$, that is if $G=N_p$ and $H=N_q$. By the structure of an independent set $S$ of $G\diamond H$ we can have whole, say $G$-layer in $S$, but then $S$ equals exactly this layer which gives $p<n(G)+n(H)-1=p+q-1$ vertices in $S$. Nevertheless, in this case it is easier to follow a direct approach which is straightforward:
$$\alpha(N_p\diamond N_q)=\alpha(K_p\times K_q)=\max\{p,q\}.$$
The bound given in Proposition~\ref{c:bound2} depends on a partition of $V(G)$ into independent sets. However, it is not necessary to consider a partition of the entire vertex set $V(G)$.  Instead, one may restrict attention to a collection of independent subsets $\{A_1,\ldots,A_t\}$ of $G$ satisfying $\sum_{i=1}^ta_i<n(G)$. This leads to the following result.

%\begin{proposition}\label{c:bound3}
 %   Let $G$ and $H$ be graphs. If ${\cal{A}}=\{A_1,\ldots ,A_t\}$ is a family of $t$ disjoint independent sets and $a_i=|A_i|$ for any $i \in [t]$, then      
%$$\alpha(G \diamond H) \leq n(H)+\max_{{\cal A},t}\Big\{\sum_{i=1}^t a_i-t\Big\}.$$
%\end{proposition}

\begin{proposition}\label{c:bound3}
    If $G$ and $H$ are graphs, then      
$$\alpha(G \diamond H) \leq n(H)+\max\Big\{\sum_{i=1}^t |A_i|-t:\, \{A_1,\ldots ,A_t\} \textrm{ is a family of disjoint independent sets of }G \Big\}.$$
\end{proposition}

\begin{proof}
    By Corollary~\ref{alpha}, there exist friendly sequences $(A_1,\ldots ,A_t)$, $(B_1,\ldots ,B_t)$ of disjoint independent sets of $G$ and $H$, respectively such that 
    $$\alpha(G \diamond H)=\sum_{i=1}^t(|A_i|+|B_i|)-t \leq \max_{{\cal A},k}\Big\{\sum_{i=1}^k |A_i|+\sum_{i=1}^k|B_i|-k\Big\} \leq$$
    $$\leq \max_{{\cal A},k}\Big\{\sum_{i=1}^k |A_i| +n(H)-k\Big\}=n(H)+\max_{{\cal A},k}\Big\{\sum_{i=1}^k|A_i|-k\Big\}.$$
\end{proof}

Next we show that the bounds of Propositions \ref{c:bound2} and \ref{c:bound3} are sharp for a special case for $\max\{\chi(G),\chi(H)\}=2$ and $t=2$, respectively.

\begin{corollary}\label{c:sumBound}
    For any graphs $G$ and $H$ of order at least two we have 
    $$\alpha(G \diamond H) \leq n(G)+n(H)-2$$ 
    and the bound is sharp if and only if one factor is $K_{\ell,1}$ and the other is $N_k$ or one factor is $K_{p,q}$ and the other $N_2$ or one factor is $N_k$ and the other $N_2$. 
\end{corollary}

\begin{proof}
To prove the bound, let ${\cal{A}}=(A_1,\ldots ,A_t)$, and ${\cal{B}}=(B_1,\ldots ,B_t)$ be friendly sequences of disjoint independent sets of $G$ and $H$, respectively, such that $\alpha(G \diamond H)=\sum_{i=1}^t(|A_i|+|B_i|)-t$, which exist by Corollary~\ref{alpha}. Let $S=\bigcup_{i=1}^t A_i \times B_i$ be the corresponding $\alpha(G\diamond H)$-set. If $t=1$, then since $\cal{A}$ and $\cal{B}$ are friendly, it follows that either $a_i=1$, or $b_i=1$ and thus  $|S| \leq \max{\{a_1,b_1\}}\leq \max{\{\alpha(G),\alpha(H)\}} \leq n(G)+n(H)-2.$ For $t \geq 2$, $\alpha(G \diamond H)=\sum_{i=1}^t(|A_i|+|B_i|)-t \leq n(G)+n(H)-2.$

For the sharpness, let $\ell$ and $k$ be arbitrary positive integers and let $G=K_{\ell,1}$ and $H=N_k$ where $g$ is the center and $\{g_1,\dots,g_{\ell}\}$ the leaves of $K_{\ell,1}$ and $V(H)=\{h_1,\dots,h_k\}$. Set $S=\{(g,h_i),(g_j,h_k):i\in [k-1],j\in[\ell]\}$ is clearly an independent set of $G\diamond H$ of cardinality $\ell+k-1=n(G)+n(H)-2$ and the bound is sharp. 

The set $S=(V_1\times\{h\})\cup(V_2\times\{h'\})$ is an independent set of $K_{p,q}\diamond N_2$ where $V_1$ and $V_2$ form the bipartition of $K_{p,q}$ and $H=N_2$ with $V(H)=\{h,h'\}$. Clearly, $|S|=|V_1|+|V_2|=p+q=n(K_{p,q})+n(N_2)-2$ and the bound is sharp again. 

Finally, for any integer $k \geq 2$, $\alpha(N_k \diamond N_2)=k = n(N_k)+n(N_2)-2.$

Conversely, assume that $\alpha(G \diamond H)=n(G)+n(H)-2$ and let $S$ be an $\alpha(G \diamond H)$-set. Let $\cal{A}$ and $\cal{B}$ be the friendly sequences obtained from $S$, as in Theorem~\ref{friends}. Let first $V(H)=\{h,h'\}$ and $\alpha(G \diamond H)=n(G)$. If $S$ equals to one $G$-layer, then $G=N_{n(G)}$ and $H$ can be arbitrary, that is either $H=K_2=K_{1,1}$ or $H=N_2$. So, assume that $S$ contains $p$ vertices of one and $q=n(G)-p$ vertices of the other $G$-layer, in which case $\cal{A}$ and $\cal{B}$ have length 2. Let by $A_1=p_G(S^h)$ and $A_2=p_G(S^{h'})$. Since $\cal{A}$ and $\cal{B}$ are friendly, $A_1$ and $A_2$ are disjoint independent set. If $H=K_2$, then there are no edges between vertices of $A_1$ and $A_2$ as $\cal{A}$ and $\cal{B}$ are friendly and $G=N_{n(G)}$. If $H=N_2$, then there are all edges between vertices of $A_1$ and $A_2$ and thus $G=K_{p,q}$. We get symmetric results when $|V(G)|=2$. 

We are left with the case where $|V(G)|>2$ and $|V(H)|>2$. By Theorem~\ref{friends} we have friendly sequences ${\cal A}=(A_1,\dots,A_t)$ and ${\cal B}=(B_1,\dots,B_t)$ such that $|S|=\bigcap_{i=1}^t A_i \times B_i$. If $t>2$, then we have a contradiction with Proposition \ref{c:bound3} because $\alpha(G \diamond H)=n(G)+n(H)-2$. 
Hence $t \leq 2$. If $t=1$, then $|S|=\max{\{a_1,b_1\}} \leq \max{\{n(G),n(H)\}} < n(G)+n(H)-2$ (as $n(G),n(H) \geq 3$), a contradiction. Thus, $t=2$. Without loss of generality let $b_1=1$. If $b_2=1$, then $|S|=a_1+a_2 \leq n(G) < n(G)+n(H)-2$, a contradiction. Thus $b_2 \geq 2$ and because $\cal{A}$ and $\cal{B}$ are friendly, $a_2=1.$ Hence $|S|=a_1+b_2$. Since $|S|=n(G)+n(H)-2$ and since $a_1\leq n(G)-1$ and $b_2\leq n(H)-1$, it follows that $a_1=n(G)-1>1$, $a_2=1$, $b_1=1$ and $b_2=n(H)-1>1$. If there exists an edge between a vertex from $A_1$ and the vertex from $A_2$, then since $\cal{A}$ and $\cal{B}$ are friendly, 
\begin{itemize}
\item there are all possible edges between $A_1$ and $A_2$ and thus $G=K_{n(G)-1,1}$,
\item there are no edges between $B_1$ and $B_2$ and hence $H=N_{n(H)}$. 
\end{itemize}

Otherwise, there are no edges between $A_1$ and $A_2$ and $G=N_{n(G)}$. Since $\cal{A}$ and $\cal{B}$ are friendly, there must exist all possible edges between $B_1$ and $B_2$ and thus $H=K_{n(H)-1,1}$.
\end{proof}

%%%%%%%%%%%%%%%%%%%%%%%%%%%%%%%%%%%%%%%%%%%%%%%%%%%%%%%%%

\section{Additional structure of independent sets of modular products}\label{sec:4}

The structure of independent sets of modular products of graphs was already described in Theorem \ref{friends}. In this section we join some sets of ${\cal A}$ and ${\cal B}$ of Theorem \ref{friends} and observe that every maximal independent set of modular product is build by some blocks of two different types (plus their commutative versions) that were already implicitly used in Theorems \ref{setA} and \ref{thm:lower1}. 

The first can be found in $K_{p_1,\dots,p_t}\diamond N_t$ and is 
$\cup_{i=1}^t(V_i\times \{h_i\})$ where $V_1,\dots,V_t$ is the partition of $V(K_{p_1,\dots,p_t})$ and $V(N_t)=\{h_1,\dots,h_t\}$. We call this set a type A set and its commutative version in $N_t\diamond K_{p_1,\dots,p_t}$ a type B set. 

The second is a subset of vertices of $K_{t,1}\diamond N_k$ and is $(\{g\}\times (V(N_k)\setminus\{h\}))\cup((V(K_{t,1})\setminus\{g\})\times \{h\})$ where $g$ is the center of $K_{t,1}$ and $h$ and arbitrary vertex from $N_k$. We call this set a type C set and its commutative version in $N_k\diamond K_{t,1}$ a type D set. 

Let $G$ and $H$ be graphs, and let $\mathcal G=\{G_1,\dots,G_k\}, \mathcal H=\{H_1,\dots,H_k\}$ be the sets of vertex disjoint induced subgraphs of $G$ and $H$, respectively. Assume that for each $i\in[k]$,
\begin{itemize}
    \item $G_i,H_i\in {\cal{F}}=\{K_{p_1,\dots,p_t},\,K_{1,t},\,N_t : t\in\mathbb N\}$, $i\in[k]$, and
    \item the modular product $G_i\diamond H_i$ yields a type A, a type B, a type C or a type D set $J_i$.
\end{itemize}
We say that sets $J_1,\dots,J_k$ are \emph{compatible} in $G\diamond H$ if $J_1\cup\cdots\cup J_k$ is an independent set of $G\diamond H$. On Figure \ref{fig1} observe graphs $G_1\cong N_2$ (red vertices), $H_1\cong K_2$ (green vertices), $G_2\cong N_2$ (blue vertices) and $H_2\cong K_{2,1}$ (yellow vertices) where $J_1$ and $J_2$ (black vertices) are compatible.

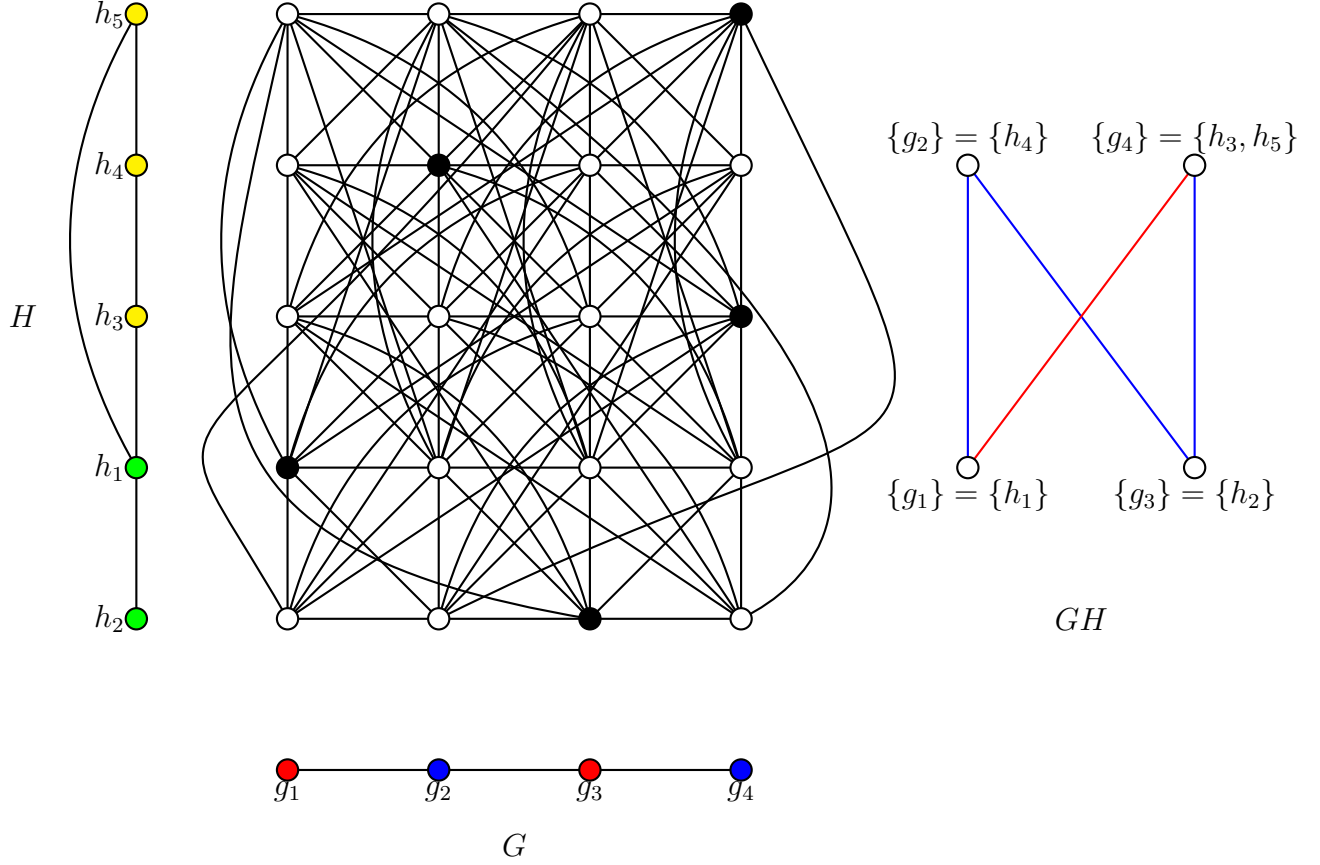
\begin{figure}[ht!]
\begin{center}
\begin{tikzpicture}[scale=1,style=thick,x=1cm,y=1cm]
\def\vr{4pt} % \vr = vertex radius;

% define vertices
%%%%%
%%%%%
\path (0,0) coordinate (a1);
\path (0,2) coordinate (a2);
\path (0,4) coordinate (a3);
\path (0,6) coordinate (a4);
\path (0,8) coordinate (a5);
\path (2,0) coordinate (b1);
\path (2,2) coordinate (b2);
\path (2,4) coordinate (b3);
\path (2,6) coordinate (b4);
\path (2,8) coordinate (b5);
\path (4,0) coordinate (c1);
\path (4,2) coordinate (c2);
\path (4,4) coordinate (c3);
\path (4,6) coordinate (c4);
\path (4,8) coordinate (c5);
\path (6,0) coordinate (d1);
\path (6,2) coordinate (d2);
\path (6,4) coordinate (d3);
\path (6,6) coordinate (d4);
\path (6,8) coordinate (d5);
\path (8,0) coordinate (e1);
\path (8,2) coordinate (e2);
\path (8,4) coordinate (e3);
\path (8,6) coordinate (e4);
\path (8,8) coordinate (e5);

\path (2,-2) coordinate (b);
\path (4,-2) coordinate (c);
\path (6,-2) coordinate (d);
\path (8,-2) coordinate (e);

\path (11,2) coordinate (x);
\path (14,2) coordinate (y);
\path (11,6) coordinate (u);
\path (14,6) coordinate (v);

%  edges
\draw (b) -- (c) -- (d) -- (e);
\draw (a1) -- (a2) -- (a3) -- (a4) -- (a5);
\draw (a2) to [bend left] (a5);
\draw (b1) -- (c1) -- (d1) -- (e1);
\draw (b1) -- (b2) -- (b3) -- (b4) -- (b5);
\draw (b2) to [bend left] (b5);
\draw (b2) -- (c2) -- (d2) -- (e2);
\draw (c1) -- (c2) -- (c3) -- (c4) -- (c5);
\draw (c2) to [bend left] (c5);
\draw (b3) -- (c3) -- (d3) -- (e3);
\draw (d1) -- (d2) -- (d3) -- (d4) -- (d5);
\draw (d2) to [bend left] (d5);
\draw (b4) -- (c4) -- (d4) -- (e4);
\draw (e1) -- (e2) -- (e3) -- (e4) -- (e5);
\draw (e2) to [bend left] (e5);
\draw (b5) -- (c5) -- (d5) -- (e5);
\draw[blue] (x)--(u) -- (y)--(v);
\draw[red] (v)--(x);

\draw (b1) -- (c2) -- (d1) -- (e2);
\draw (b2) -- (c1) -- (d2) -- (e1);
\draw (b3) -- (c2) -- (d3) -- (e2);
\draw (b2) -- (c3) -- (d2) -- (e3);
\draw (b3) -- (c4) -- (d3) -- (e4);
\draw (b4) -- (c3) -- (d4) -- (e3);
\draw (b4) -- (c5) -- (d4) -- (e5);
\draw (b5) -- (c4) -- (d5) -- (e4);
\draw (b2) -- (c5) -- (d2) -- (e5);
\draw (b5) -- (c2) -- (d5) -- (e2);
\draw (b1) -- (d4);
\draw (b4) -- (d1);
\draw (c1) -- (e4);
\draw (c4) -- (e1);

\draw (b1) to [bend left] (d3);
\draw (b3) to [bend left] (d1);
\draw (c1) to [bend left] (e3);
\draw (c3) to [bend left] (e1);
\draw (b1) .. controls (0,3.5) and (0.3,1).. (d5);
\draw (c1) .. controls (11,3.5) and (11.3,1).. (e5);
\draw (b5) .. controls (1,3.5) and (0,1).. (d1);
\draw (c5) .. controls (10,3.5) and (10,1).. (e1);
\draw (b2) to [bend left] (d4);
\draw (b4) to [bend left] (d2);
\draw (c2) to [bend left] (e4);
\draw (c4) to [bend left] (e2);
\draw (b3) to [bend left] (d5);
\draw (b5) to [bend left] (d3);
\draw (c3) to [bend left] (e5);
\draw (c5) to [bend left] (e3);

\draw (b2) to (e4);
\draw (b4) to (e2);
\draw (b1) to (e3);
\draw (b3) to (e1);
\draw (b3) to (e5);
\draw (b5) to (e3);

%\draw (c) -- (f);

\draw (a1) [fill=green] circle (\vr); \draw (a2) [fill=green] circle (\vr); \draw (a3) [fill=yellow] circle (\vr);
\draw (a4) [fill=yellow] circle (\vr); \draw (a5) [fill=yellow] circle (\vr);
\draw (b1) [fill=white] circle (\vr); \draw (b2) [fill=black] circle (\vr); \draw (b3) [fill=white] circle (\vr);
\draw (b4) [fill=white] circle (\vr); \draw (b5) [fill=white] circle (\vr);
\draw (c1) [fill=white] circle (\vr); \draw (c2) [fill=white] circle (\vr); \draw (c3) [fill=white] circle (\vr);
\draw (c4) [fill=black] circle (\vr); \draw (c5) [fill=white] circle (\vr);
\draw (d1) [fill=black] circle (\vr); \draw (d2) [fill=white] circle (\vr); \draw (d3) [fill=white] circle (\vr);
\draw (d4) [fill=white] circle (\vr); \draw (d5) [fill=white] circle (\vr);
\draw (e1) [fill=white] circle (\vr); \draw (e2) [fill=white] circle (\vr); \draw (e3) [fill=black] circle (\vr);
\draw (e4) [fill=white] circle (\vr); \draw (e5) [fill=black] circle (\vr);
\draw (b) [fill=red] circle (\vr);
\draw (c) [fill=blue] circle (\vr);
\draw (d) [fill=red] circle (\vr);
\draw (e) [fill=blue] circle (\vr);

\draw (x) [fill=white] circle (\vr);
\draw (y) [fill=white] circle (\vr);
\draw (u) [fill=white] circle (\vr);
\draw (v) [fill=white] circle (\vr);

\draw[anchor = north] (b) node {$g_1$};
\draw[anchor = north] (c) node {$g_2$};
\draw[anchor = north] (d) node {$g_3$};
\draw[anchor = north] (e) node {$g_4$};
\draw[anchor = east] (a1) node {$h_2$};
\draw[anchor = east] (a2) node {$h_1$};
\draw[anchor = east] (a3) node {$h_3$};
\draw[anchor = east] (a4) node {$h_4$};
\draw[anchor = east] (a5) node {$h_5$};
\draw[anchor = north] (x) node {$\{g_1\}=\{h_1\}$};
\draw[anchor = north] (y) node {$\{g_3\}=\{h_2\}$};
\draw[anchor = south] (u) node {$\{g_2\}=\{h_4\}$};
\draw[anchor = south] (v) node {$\{g_4\}=\{h_3,h_5\}$};
\draw (12.5,0) node {$GH$};

\draw (-1.5,4) node {$H$};
\draw (5,-3) node {$G$};

\end{tikzpicture}
\end{center}
\caption{Left: independent set of $G\diamond H$ represent black vertices, subgraphs $G_1\cong N_2$ with vertices $g_1,g_3$ (red vertices) and $G_2\cong N_2$ with vertices $g_2,g_4$ (blue vertices), $H_1\cong K_2$ with vertices $h_1, h_2$ (green vertices) and $H_2\cong K_{1,2}$ with vertices $h_3,h_4,h_5$ (yellow vertices). Right: $(GH)_{\mathcal{GH}}$, where ${\mathcal{G}}=\{N_2,N_2\}$ and ${\mathcal{H}}=\{K_2,K_{1,2}\}$.} \label{fig1}
\end{figure}

%%%%%%%%%%%%%%%%%%%%%%%%%%%%%%%%%%%%%%%%%%%%%%%%%%%%%%%%%%%%%%%%%%%%%%%%%%%%%%%%%
%%%%%%%%%%%%%%%%%%%%%%%%%%%%%%%%%%%%%%%%%%%%%%%%%%%%%%%%%%%%%%%%%%%%%%%%%%%%%%%%%

Next we define the \emph{partition graph} $(GH)_{\mathcal{GH}}$ of graphs $G$ and $H$ as follows.
The vertices of $(GH)_{\mathcal{GH}}$ are \emph{blocks}, i.e.\ certain subsets of
$V(G)$ and $V(H)$. Formally,
\[
V\bigl((GH)_{\mathcal{GH}}\bigr)=\bigcup_{i=1}^k \mathcal W_i,
\]
where the families $\mathcal W_1,\dots,\mathcal W_k$ are pairwise disjoint.
Each block $X\in V((GH)_{\mathcal{GH}})$ is equipped with two representations $\phi_G(X)\subseteq V(G)$ and $\phi_H(X)\subseteq V(H)$,
called the \emph{$G$-representation} and \emph{$H$-representation} of $X$, respectively. For fix $i\in[k]$, the family $\mathcal W_i$ is defined according to the type of $J_i$.

\begin{itemize}
\item[\textup{(A)}]
If $J_i$ is of type $A$, or equivalently if $G_i\cong K_{p_1,\dots,p_s}$ and $H_i\cong N_s$, then let $\{V_{i,1},\dots,V_{i,s}\}$ be the partition of $V(G_i)$ and
$V(H_i)=\{h_{i,1},\dots,h_{i,s}\}$.
Define $\mathcal W_i:=\{X_{i,1},\dots,X_{i,s}\}$, where for each $j\in[s]$,
$\phi_G(X_{i,j})=V_{i,j}, \phi_H(X_{i,j})=\{h_{i,j}\}$.

\item[\textup{(B)}]
If $J_i$ is of type $B$ ($G_i\cong N_s$ and $H_i\cong K_{p_1,\dots,p_s}$), then
let $V(G_i)=\{g_{i,1},\dots,g_{i,s}\}$ and let
$\{V_{i,1},\dots,V_{i,s}\}$ be the partition of $V(H_i)$.
Define $\mathcal W_i:=\{X_{i,1},\dots,X_{i,s}\}$,
where for each $j\in[s]$, $\phi_G(X_{i,j})=\{g_{i,j}\},
\phi_H(X_{i,j})=V_{i,j}$.

\item[\textup{(C)}]
If $J_i$ is of type $C$ ($G_i\cong K_{r,1}$ and $H_i\cong N_s$), then
let $g_i$ be the center of $G_i$ and $L_i$ the set of its leaves, and let
$V(H_i)=\{h_{i,1},\dots,h_{i,s}\}$.
Define $\mathcal W_i:=\{X_{i,1},X_{i,2}\}$, where $\phi_G(X_{i,1})=\{g_i\}, 
\phi_H(X_{i,1})=\{h_{i,1},\dots,h_{i,s-1}\}$, and $\phi_G(X_{i,2})=L_i, \phi_H(X_{i,2})=\{h_{i,s}\}$.

\item[\textup{(D)}]
If $J_i$ is of type $D$ ($G_i\cong N_s$ and $H_i\cong K_{r,1}$), then
let $h_i$ be the center of $H_i$ and $L_i$ the set of its leaves, and let
$V(G_i)=\{g_{i,1},\dots,g_{i,s}\}$.
Define $\mathcal W_i:=\{X_{i,1},X_{i,2}\}$, where $\phi_G(X_{i,1})=\{g_{i,1},\dots,g_{i,s-1}\}, \phi_H(X_{i,1})=\{h_i\}$, and $\phi_G(X_{i,2})=\{g_{i,s}\}, \phi_H(X_{i,2})=L_i$.
\end{itemize}

It remains to define the edges of the partition graph $(GH)_{\mathcal{GH}}$.
Let $X\in\mathcal W_i$ and $Y\in\mathcal W_j$ with $i\neq j$.
We put $XY\in E((GH)_{\mathcal{GH}})$ if and only if exactly one of the following holds:
\begin{enumerate}
\item every vertex of $\phi_G(X)$ is adjacent in $G$ to every vertex of $\phi_G(Y)$
and no vertex of $\phi_H(X)$ is adjacent in $H$ to any vertex of $\phi_H(Y)$;
\item every vertex of $\phi_H(X)$ is adjacent in $H$ to every vertex of $\phi_H(Y)$
and no vertex of $\phi_G(X)$ is adjacent in $G$ to any vertex of $\phi_G(Y)$.
\end{enumerate}
No edges are present between vertices belonging to the same family $\mathcal W_i$. See the right side of Figure \ref{fig1} for $(GH)_{\cal{GH}}$ where blue edges represent appropriate edges in $G$ and non-edges in $H$, while the red edge represents appropriate edges in $H$ and no-edges in $G$.

\begin{theorem}\label{compatible}
Let $G$ and $H$ be two graphs with their appropriate families of disjoint induced subgraphs ${\cal{G}}=\{G_1,\dots,G_k\}, {\cal{H}}=\{H_1,\dots,H_k\}$,
with $G_i,H_i\in\{K_{p_1,\dots,p_t},K_{1,t},N_t: t \in {\mathbb{N}}\}$ for every $i\in[k]$ such that $J_i$ is an independent set of $G_i \diamond H_i$ either of type $A,B,C$ or $D$ for $i \in [k]$. Independent sets $J_1,\dots,J_k$ are compatible if and only if $(GH)_{\mathcal{GH}}$ is a complete $k$-partite graph.
\end{theorem}

\begin{proof}
Let ${\mathcal{G}}=\{G_1,\ldots , G_k\}$ and ${\cal{H}}=\{H_1,\ldots ,H_k\}$ be appropriate vertex disjoint subgraphs of $G$ and $H$, respectively ($G_i,H_i\in\{K_{p_1,\dots,p_t},K_{1,t},N_t; t \in {\mathbb{N}}\}$). Moreover, for any $i \in [k]$, let $J_i$ be an independent set of $G_i \diamond H_i$ of the type $A,B,C$ or $D$, depending on the structure of $G_i$ and $H_i$. Since for any $i \in [k]$, $J_i$ is an independent set of $G_i \diamond H_i$, Theorem~\ref{friends}
 implies that there exist friendly sequences ${\cal{A}}_i=(A_{i,1},\ldots , A_{i,{n_i}})$ and ${\cal{B}}_i=(B_{i,1},\ldots , B_{i,{n_i}})$ of $G_i$ and $H_i$, respectively.
 
Assume first that $J_1,\ldots , J_k$ are compatible. Hence $J=J_1\cup \cdots \cup J_k$ is an independent set of $G \diamond H$. Hence ${\cal{A}}=(A_{1,1},\ldots , A_{1,n_1},\ldots ,A_{k,1},\ldots , A_{k,n_k})$  and ${\cal{B}}=(B_{1,1},\ldots , B_{1,n_1},\ldots ,B_{k,1},\ldots , B_{k,n_k})$ are friendly sequences of $G$ and $H$, respectively. Moreover, Observation~\ref{quotient} implies that $G$-quotient graph $G_J$ obtained from $\cal{A}$ is isomorphic to the complement of the $H$-quotient graph $H_J$ obtained from $\cal{B}$. Now, consider the partition graph $(GH)_{\mathcal{GH}}$. By the definition of the partition graph, it follows that $V((GH)_{\mathcal{GH}})=\bigcup_{i\in [k]}\{X_{i,1},\ldots ,X_{i,n_i}\}$, where the $G$-representation of $X_{i,j}$ is $A_{i,j}$ and its $H$-representation is $B_{i,j}$ for any $i \in [k], j \in [n_i]$. Moreover, the definition of the representation graph implies that $X_{i_1,j_1}X_{i_2,j_2} \in E((GH)_{\mathcal{GH}})$ if and only if $i_1 \neq i_2$ and either there are all possible edges between $A_{i_1,j_1}$ and $A_{i_2,j_2}$ and no edge between $B_{i_1,j_1}$ and $B_{i_2,j_2}$ or vice versa. Since $G_J \cong \overline{H}_J$, it follows that either $A_{i_1,j_1}A_{i_2,j_2} \in E(G_J)$ and $B_{i_1,j_1}B_{i_2,j_2} \notin E(H_J)$ or $A_{i_1,j_1}A_{i_2,j_2} \notin E(G_J)$ and $B_{i_1,j_1}B_{i_2,j_2} \in E(H_J)$. Both cases implies that $X_{i_1,j_1}X_{i_2,j_2} \in E((GH)_{\mathcal{GH}})$ if and only if $i_1 \neq i_2$. Hence $(GH)_{\mathcal{GH}}$ is a complete $k$-partite graph.

For the converse assume that $(GH)_{\mathcal{GH}}$ is a complete $k$-partite graph with $V((GH)_{\mathcal{GH}}) =\bigcup_{i\in [k]}\{X_{i,1},\ldots , X_{i,n_i}\}$ where $X_{i_1,j_1}$ is adjacent to $X_{i_2,j_2}$ if and only if $i_1\neq i_2$. For the purpose of contradiction assume that $J_1,\ldots , J_k$ are not compatible. Thus the set $J=J_1\cup \ldots \cup J_k$ is not an independent set of $G \diamond H$. Since for any $i \in [k]$, $J_i$ is an independent set of $G_i \diamond H_i \subseteq G \diamond H$, it follows that there exist $i_1,i_2 \in [k]$, $i_1 \neq i_2$, such that $(g_1,h_1) \in J_{i_1}$ and $(g_2,h_2) \in J_{i_2}$ are adjacent in $G \diamond H$. Since $(g_1,h_1) \in J_{i_1} \subseteq V(G_{i_1} \diamond H_{i_1})$, $(g_2,h_2) \in J_{i_2} \subseteq V(G_{i_2} \diamond H_{i_2})$ and since $V(G_{i_1}\diamond H_{i_1}) \cap V(G_{i_2} \diamond H_{i_2})=\emptyset$, it follows that $g_1\neq g_2$ and $h_1 \neq h_2$. Hence adjacency of $(g_1,h_1)$ and $(g_2,h_2)$ in $G \diamond H$ implies that either $g_1g_2 \in E(G)$ and $h_1h_2\in E(H)$ or $g_1g_2 \notin E(G)$ and $h_1h_2 \notin E(H)$. For $\ell \in [2]$, let $V_{i_\ell,j_\ell}$ be a $G$-representation of $X_{i_\ell,j_\ell}$ that contains $g_\ell$, and let  $U_{i_\ell,j_\ell}$ be an $H$-representation of $X_{i_\ell,j_\ell}$ that contains $h_\ell$. Since either $g_1g_2 \in E(G), h_1h_2 \in E(H)$ or  $g_1g_2 \notin E(G), h_1h_2 \notin E(H)$ the definition of the edge set of the partition graph implies that $X_{i_1,j_1}X_{i_2,j_2} \notin E((GH)_{\mathcal{GH}})$. Now $i_1 \neq i_2$ yields is a contradiction, because two vertices $X_{i_1,j_1},X_{i_2,j_2}$ from $V((GH)_{\mathcal{GH}})$ are adjacent if and only if $i_1 \neq i_2$.
\end{proof}

Theorem \ref{compatible} together with Theorem \ref{friends} guaranties us that every independent set of a modular product is build by some blocks of type A, B, C and D (there can be more of the same type). While it is easy to build graphs under the conditions of Theorem \ref{compatible} that have independent set with many compatible blocks, it seems quite difficult to describe maximum independent set of $G\diamond H$ when the factors are given. Nevertheless, this task may be more manageable if we restrict to certain graph classes. Graphs of girth at least five seems to be such an example. (Recall that the girth $g(G)$ of a graph $G$ is the length of a shortest cycle.) The reason is that only multipartite subgraphs in a graph $G$ with $g(G)\geq 5$ are stars $K_{k,1}$ and one-partite graph $N_k$. Hence type A  immediately reduce from $K_{p_1,\dots,p_t}\diamond N_t$ to  $N_{p_1}\diamond K_1$ where $K_1$ is a vertex $h\in V(H)$ (and symmetrically also type B). Moreover, compatibility of independent sets in $G\diamond H$ often yields cycles $C_3$ and $C_4$, which is not possible since $g(G)\geq 5$. So, we end this discussion with the following open problem.

\begin{problem}
Determine $\alpha(G\diamond H)$ if $g(G)\geq 5$ and $g(H)\geq 5$.
\end{problem}

%%%%%%%%%%%%%%%%%%%%%%%%%%%%%%%%%%%%%%%%%%%%%%%%%%%%%%%%%%%%%%%%%%%%%%%%%%%%%%%%%%%%%%%%%%%%%%%%%%%%%%%%%%%%%%%%%%%%%%%%%%%%%%%%%%%%%%%%%%%%%%%%%%%%%%%%
%%%%%%%%%%%%%%%%%%%%%%%%%%%%%%%%%%%%%%%%%%%%%%%%%%%%%%%%%%%%%%%%%%%%%%%%%%%%%%%%%%%%%%%%%%%%%%%%%%%%%%%%%%%%%%%%%%%%%%%%%%%%%%%%%%%%%%%%%%%%%%%%%%%%%%

\section*{Acknowledgements}
This work has been supported by the European Commission's Horizon Europe Research
and Innovation programme through the Marie Skłodowska-Curie Actions Staff Exchanges
(MSCA-SE) under Grant Agreement no.101182819 (COVER: (C)ombinatorial (O)ptimization for
(V)ersatile Applications to (E)merging u(R)ban Problems)
The authors also acknowledge the financial support of the Slovenian Research and Innovation Agency (research core funding No.\ P1-0297 and project N1-0431).

%%%%%%%%%%%%%%%%%%%%%%%%%%%%%%%%%%%%%%%%%%%%%%%%%%%%%%%%%%%%%%%%%%%%%%%

\end{document}